\documentclass[11pt]{article}
\usepackage{amsthm}
\usepackage{amsmath}
\usepackage{amssymb}
\usepackage{gensymb}
\usepackage{amsfonts}
\usepackage{svg}
\usepackage{graphicx}
\usepackage[ruled]{algorithm}
\usepackage[noend]{algpseudocode}
\usepackage{enumerate}
\usepackage{todonotes}
\usepackage{mathtools}
\usepackage{subcaption}
\mathtoolsset{showonlyrefs}
\usepackage{comment}
\usepackage{tikz-cd}
\usepackage{caption,subcaption}
\usepackage{soul}
\usepackage[makeroom]{cancel}
\usepackage{setspace}
\usepackage{wrapfig}

\definecolor{darkblue}{rgb}{0.0, 0.0, 0.8}

\usepackage[colorlinks=true, allcolors=darkblue]{hyperref}

\usepackage{array}
\newcommand{\PreserveBackslash}[1]{\let\temp=\\#1\let\\=\temp}
\newcolumntype{C}[1]{>{\PreserveBackslash\centering}p{#1}}
\newcolumntype{R}[1]{>{\PreserveBackslash\raggedleft}p{#1}}
\newcolumntype{L}[1]{>{\PreserveBackslash\raggedright}p{#1}}

\newtheorem{claim}{Claim}
\newtheorem{proposition}{Prop.}
\newtheorem{thm}{Theorem}
\newtheorem{lemma}{Lemma}

\newtheorem{remark}{Remark}
\newtheorem{definition}{Definition}

\newcommand{\R}{\mathbb{R}}

\newcommand{\Imm}{\operatorname{Imm}}

\newcommand{\on}[1]{\operatorname{#1}}
\newcommand{\gsi}{{g}}
\newcommand{\sgsi}{{G}}
\newcommand{\gcc}{\widetilde{g}}
\newcommand{\sgcc}{\widetilde{G}}
\newcommand{\Rdnst}{\R^{d\times n}_*}
\newcommand{\Rdn}{\R^{d\times n}}
\newcommand{\pathofcurves}{\gamma}
\usepackage[margin=1in]{geometry}

\title{Sobolev Metrics on  Spaces of Discrete Regular Curves }
\author{Jonathan Cerqueira, Emmanuel Hartman, Eric Klassen and Martin Bauer}
\begin{document}
\maketitle
\begin{abstract}
Reparametrization invariant Sobolev metrics on spaces of regular curves have been shown to be of importance  in the field of mathematical shape analysis. For practical applications, one usually discretizes the space of smooth curves and considers the induced Riemannian metric  on a finite dimensional approximation space. Surprisingly, the theoretical properties of the corresponding finite dimensional Riemannian manifolds have not yet been studied in detail, which is the content of the present article. Our main theorem concerns metric and geodesic completeness and mirrors the results of the infinite dimensional setting as obtained by Bruveris, Michor and Mumford. 
\end{abstract}

\tableofcontents

\section{Introduction}
\paragraph{Motivation and Background:}
Reparametrization invariant Sobolev metrics on spaces of regular curves play a central role in the field of mathematical shape analysis. Due to their reparametrization invariance, these metrics descend to Riemannian metrics on  spaces of unparametrized curves, which are of relevance in mathematical shape analysis and data analysis. Examples include applications where one is interested  in the shape of planar objects (represented by their boundary curves); see \cite{srivastava2016functional,younes2010shapes,michor2007overview,bauer2014overview} and the references therein. Motivated by their appearance in these applications there has been a large interest in studying their mathematical properties. Michor and Mumford~\cite{michor2005vanishing,mumford2006riemannian,bauer2012vanishing} showed a surprising degeneracy of the simplest such metric; namely they proved that the reparametrization invariant $L^2$-metric, i.e., the Sobolev metric of order zero, induces a degenerate distance function. This purely infinite dimensional phenomenon renders this metric unsuited for mathematical shape analysis, as it assigns a zero distance between any two curves and thus cannot distinguish between different shapes. For higher order metrics this degeneracy disappears and it has been shown in many applications that they lead  to meaningful notions of distance~\cite{srivastava2010shape,sundaramoorthi2007sobolev,needham2020simplifying,bauer2017numerical,celledoni2018shape}. As a result, these metrics can be  used to define a mathematical framework for statistical shape analysis on these spaces of curves~\cite{pennec2019riemannian}. A natural question that arises in this context concerns the existence of minimizing geodesics, i.e., whether the space of  regular curves equipped with these Riemannian metrics is a geodesically complete and/or geodesically convex space. For metrics of order two or higher a positive  answer to this question has been found by Bruveris, Michor and Mumford~\cite{bruveris2014geodesic,bruveris2017completeness}; more recently, it has been shown by one of the authors and collaborators~\cite{bauer2024completeness} that 3/2 is actually the critical index for this property, i.e., for a Sobolev metric of order greater than 3/2 the resulting space is geodesically complete, whereas there always exists geodesics that leave the space in finite time if the order is smaller than 3/2. The behavior at the critical value 3/2 is still open.

\paragraph{Main Contributions:}
For practical applications, one usually discretizes the space of smooth curves and considers the induced Riemannian metric  in a corresponding finite dimensional approximation space. Discretizations that  have been considered include approximating curves via piecewise linear functions~\cite{bernal2016fast,srivastava2016functional,bauer2018soliton},
B-spline discretizations~\cite{bauer2017numerical} or finite Fourier series approximations~\cite{beutler2024discrete}.  In this paper, we will discretize a curve as a finite sequence of points in Euclidean space, so that the space of curves is the space of these sequences. 
Using methods of discrete differential geometry, we will define a class of metrics on this finite dimensional space that are motivated by and analogous to the class of reparametrization invariant Sobolev metrics mentioned above for the infinite dimensional space of smooth curves.  To our surprise, these rather natural finite dimensional Riemannian manifolds have not yet been studied in much detail, with the only exception being the case of the homogenous Sobolev metric of order one, where the space of PL curves can be viewed as a totally geodesic submanifold of the infinite dimensional setting~\cite{bauer2018soliton};  note that a similar result for more general Sobolev metrics  is not true. 

Considering these finite dimensional approximation manifolds leads to a natural question, which is the starting point of the present article: 
\begin{center}
\emph{Which properties of the infinite dimensional geometry are mirrored in these finite dimensional geometries?}
\end{center}
As vanishing geodesic distance is a purely infinite dimensional phenomenon — every finite dimensional Riemannian manifold admits a non-degenerate geodesic distance function~\cite{lang1972differential} —  one cannot hope to observe the analogue of this result in our setting. The main result of the present article shows, however, that some of these discretizations do indeed capture the aforementioned completeness properties, cf.~Theorem~\ref{mainbit}. In addition to these theoretical results, we present in Section~\ref{sec:numerics} selected numerical examples showcasing the effects of the order of the metric on the resulting geodesics. Finally, for the special case of triangles in the plane, we study the Riemannian curvature of the space of triangles and observe that it explodes near the singularities of the space, i.e., where two points of the triangle come together. 

\paragraph{Conclusions and future work:} In this article we studied discrete Sobolev type metrics on the space of discrete regular curves in Euclidean space (where we identified this space as a space of sequences of points) and showed that these geometries mirror several properties of their infinite dimensional counterparts. In future work we envision several distinct research directions: first, we have restricted ourselves in the present study to integer order Sobolev metrics. In future work it would be interesting to perform a similar analysis also for the class of fractional (real) order Sobolev metrics, such as those studied in~\cite{bauer2024completeness}.
Secondly, we aim to study stochastic completeness of these geometries.  For extrinsic metrics on the  two-landmark space it has  recently been shown by Habermann, Harms and Sommer~\cite{habermann2024long} that the resulting space is stochastically complete, assuming certain conditions on the kernel function. We believe that a similar approach could be  applied successfully to the geometries of the present article, which would be of interest in several applications where stochastic processes on shape spaces play a central role. Finally, we would like to study similar questions in the context of reparametrization invariant metrics on the space of surfaces: in this case geodesic completeness in the smooth category, i.e., in the infinite dimensional setting, is wide open and we hope to get new insights for this extremely difficult open problem by studying its finite dimensional counterpart. 

\paragraph{Acknowledgements:} M.B was partially
supported by NSF grants DMS–2324962, DMS-1953244 and CISE 2426549. M.B and J.C. were partially supported by the BSF under grant 2022076. E.H. was partially supported by NSF grant DMS-2402555.
\section{Reparametrization invariant Sobolev metrics on the space of smooth curves}
In this section we will recall some basic definitions and results regarding the class of reparametrization invariant Sobolev metrics on the space of smooth, regular (immersed) curves. 
We begin by defining the set of smooth immersions of the circle $S^1$ into the space $\R^d$:
$$\Imm(S^1,\R^d)=\{c\in C^\infty(S^1, \R^d) : |c'(\theta)|\ne 0, \forall \theta\in S^1\}.$$
Here we identify $S^1$ with the interval $[0,1]$ with its ends identified. The set $\Imm(S^1,\R^d)$ is an open subset of the Fr\'echet space $C^\infty(S^1,\R^d)$ and thus can be considered as an infinite dimensional Fr\'echet manifold using a single chart.
We let $h,k$ denote our tangent vectors, which belong to
$$T_c\on{Imm}(S^1,\R^d)\cong C^\infty(S^1,\R^d).$$
Next, we consider the space of orientation-preserving smooth self-diffeomorphisms of the circle $$\on{Diff}(S^1)=\{\varphi\in C^\infty(S^1,S^1):\varphi\text{ is bijective and } \varphi'(\theta)>0, \forall\theta\in S^1\},$$
which is an infinite dimensional Fr\'echet Lie group and acts on 
$\Imm(S^1,\R^d)$ from the right 
via the map $(c,\varphi)\mapsto c\circ\varphi$.

The principal goal of the present article is to study Riemannian geometries on $\Imm(S^1,\R^d)$. To introduce our class of Riemannian metrics we will first need to introduce some additional notation: we  denote by $D_\theta$ the derivative with respect to $\theta$ and let $D_s=\frac{1}{|c'(\theta)|}D_\theta$ and $ds=|c'(\theta)|d\theta$ be arc-length differentiation and integration respectively. Furthermore, we let $\ell(c):=\int_{S^1}|c'(\theta)|d\theta$ denote the  length of a curve $c$. Using these notations, we are ready to define the $m$-th order, reparametrization invariant Sobolev metric on $\Imm(S^1,\R^d)$:
\begin{definition}[Reparametrization invariant Sobolev Metrics]
  Let $c\in\Imm(S^1,\R^d)$, $m\in \mathbb{Z}_{\geq 0}$, and $h,k\in T_c\Imm(S^1,\R^d)$. The $m$-th order Sobolev metric on $\Imm(S^1,\R^d)$ is then given by
  \begin{equation}\label{smoothmetricdef}
    \sgsi^m_c(h,k):=\int_{S^1}\frac{\langle h,k\rangle}{\ell(c)^3}+\ell(c)^{-3+2m}\langle D_s^m h,D_s^m k\rangle ds.
  \end{equation}
\end{definition}
\begin{remark}[Scale invariant vs. non-scale invariant metrics]
In the above definition we used length dependent weights which allowed us to define a scale-invariant version of the Sobolev metric of order $m$. Alternatively we could have considered the constant coefficient Sobolev metric
  \begin{equation}
    \sgsi^m_c(h,k):=\int_{S^1}\langle h,k\rangle+\langle D_s^m h,D_s^m k\rangle ds.
  \end{equation}
and its discrete counterpart.
Almost all of the results of the present article hold also for this class of metrics, albeit with minor adaptions in the proof of the main completeness result, cf. Appendix~\ref{sec:appendix:constant}.
\end{remark}

We start by collecting several useful properties of the above defined family of Riemannian metrics:
\begin{lemma}\label{smoothproperties}
  Let $m\geq 0$ and let $\sgsi^m$ be the Riemannian metric as defined in~\eqref{smoothmetricdef}. Let $c\in \Imm(S^1,\R^d)$ and $h,k\in T_c\Imm(S^1,\R^d)$. Then we have:
    \begin{enumerate}
  \item The metric $\sgsi^m$ is invariant with respect to reparametrizations, rescalings, rotations, and translations; i.e. for $\varphi\in\on{Diff}(S^1)$, $\lambda\in\R^+$, $R\in \on{SO}(\R^d)$ and $\mathbf{v}\in\R^d$ we have 
    \begin{align*} 
      \sgsi_c^m(h,k)&=\sgsi_{c\circ\varphi}^m(h\circ\varphi,k\circ\varphi)=\sgsi_{\lambda c}^m(\lambda h,\lambda k)=\sgsi_{c+\mathbf{v}}^m(h,k)=\sgsi_{Rc}^m(Rh,Rk).
    \end{align*}
    We note in the above the action of translation $(+\mathbf{v})$ only affects the tangent space in which $h,k$ lie, but not the vectors $h$ and $k$, since the derivative of translation is the identity map.
  \item The metric $\sgsi^m$ is equivalent to any metric of the form
    $$\widetilde{\sgsi}^m(h,k):=\int_{S^1}\sum_{j=0}^m a_j\ell(c)^{-3+2j}\langle D_s^j h,D_s^j k\rangle ds$$ for $a_0,a_m>0$ and $a_j\ge 0$ for $j=1,\ldots,m-1$; i.e., there exists a $C>0$ such that for all $c\in \Imm(S^1,\R^d)$ and all $h\in T_c\Imm(S^1,\R^d)$
    $$\frac{1}{C}\widetilde{\sgsi}^m_c(h,h)\le \sgsi^m_c(h,h)\le C \widetilde{\sgsi}^m_c(h,h).$$
  \end{enumerate}
\end{lemma}

\begin{proof}
The proof of the invariances follows by direct computation. The second statement is implied by \cite[Lemma 4.2]{bauer2022sobolev} in a similar way as in Lemma \ref{discreteinvariances} below.\\
\end{proof}
For a Riemannian metric $g$ on a smooth manifold $\mathcal{M}$ one defines the {\it induced path length}, i.e., for $\gamma(t):[0,1]\to\mathcal{M}$ we let 
$$\L_{g}(\gamma)=\int_0^1\sqrt{g_{\gamma(t)}(D_t\gamma(t),D_t\gamma(t))}dt.$$
Using this one can consider the induced geodesic distance, which is defined via
$$d_{g}(p,q)=\inf \L_g(\gamma)$$
where the infimum is calculated over the set of piecewise $C^\infty$ paths $[0,1]\to\mathcal{M}$ with $\gamma(0)=p$ and $\gamma(1)=q$.
In finite dimensions this always defines a true distance function; for infinite dimensional manifolds, however, one may have a degenerate distance function, i.e., there may be distinct points $c_0,c_1\in \mathcal{M}$ such that $d_g(c_0,c_1)=0$, see for example~\cite{michor2005vanishing,bauer2020vanishing}. For the class of Sobolev metrics on spaces of curves this phenomenon has been studied by Michor and Mumford and collaborators and a full characterization of the degeneracy has been obtained:
\begin{lemma}[\cite{michor2005vanishing,bauer2024completeness}]
 The metric $\sgsi$ defines a non-degenerate geodesic distance function on the space of curves $\Imm(S^1,\mathbb R^d)$ if and only if $m\geq 1$.
\end{lemma}

The main focus on this article concerns geodesic and metric completeness properties of the corresponding space. Recall that 
a Riemannian manifold $(\mathcal{M},g)$ is \textbf{metrically complete} if it is complete as a metric space under the distance function given by the Riemannian metric. Furthermore it is called \textbf{geodesically complete} if the geodesic equation has solutions defined for all time for any initial conditions $\gamma(0)=p\in\mathcal M$ and $\gamma'(0)=v\in T_p\mathcal{M}$ and it is called  \textbf{geodesically convex} if for any two points $p,q\in\mathcal{M}$ there exists a length minimizing path connecting them. In finite dimensions  the theorem of Hopf-Rinow implies that metric and geodesic completeness are equivalent and either implies geodesic convexity, but this result famously does not hold in infinite dimensions \cite{atkin1975hopf,ekeland1978hopf,atkin1997geodesic}. 

The following theorem will characterize the metric and geodesic completeness properties of $(\Imm(S^1,\R^d),\sgsi^m)$.
To state the theorem, we will first need to introduce the space of regular curves of finite (Sobolev) regularity, i.e., for $m\geq 2$ we let
$$\mathcal{I}^m(S^1,\R^d)=\{c\in H^m(S^1,\R^d): |c'(\theta)|\ne0,\forall\theta\}.$$ 
\begin{thm}[\cite{bauer2022sobolev,bruveris2017completeness}]\label{smoothcomplete}
  Let $m\ge 0$, $c\in\Imm(S^1,\R^d)$, $h,k\in T_c\Imm(S^1,\R^d)$ and $G^m$ as defined above. Then 
  $(\Imm(S^1,\R^d),\sgsi^m)$ is a Riemannian manifold and for $m\ge 2$ the following hold:
  \begin{enumerate}[1.]
  \item The manifold $(\Imm(S^1,\R^d),G^m)$ is geodesically complete.
  \item The metric completion of $(\Imm(S^1,\R^d),G^m)$ is $(\mathcal{I}^m,G^m)$.
  \item The metric completion of $(\Imm(S^1,\R^d),G^m)$ is geodesically convex.
  \end{enumerate}
\end{thm}
\begin{proof}
  Proof of the above three statements can be found in \cite{bauer2022sobolev}, specifically via the proofs of Theorems 5.1, 5.2, and 5.3.
\end{proof}

Note that $(\Imm(S^1,\R^d),G^m)$ is not metrically complete for any $m$ (this follows from Statement 2 of the Theorem) despite being geodesically complete for $m\ge2$.
\begin{remark}[Fractional Order Metrics]
In the exposition of the present article, we have only focused on integer order Sobolev metrics. More recently, the equivalent of these results have been also shown for fractional (real) order Sobolev metrics, where the critical order for positivity of the geodesic distance is $\frac12$ and the critical order for completeness is $\frac32$, see~\cite{bauer2024completeness,bauer2018fractional} for more details. 
\end{remark}

\section{Discrete Sobolev Metrics on Discrete Regular Curves}
In this section we will introduce the main concept of the present article: a discrete version of the reparametrization invariant Sobolev metric.
We start this section by first defining a natural discretization of the space of immersed curves and then introducing a discrete analogue of $\sgsi^m$ that, under modest assumptions, converges to its smooth counterpart. Let $\Rdn$ denote the set of ordered $n$-tuples of points in $\R^d$ and consider the subset
$$\Rdnst=\left\{(\mathbf{x}_{1},\ldots, \mathbf{x}_{n})\in \Rdn: \mathbf{x}_i\in\R^d \text{ and }\mathbf{x}_i\neq \mathbf{x}_{i+1}, \forall  i\in \frac{\mathbb{Z}}{n\mathbb{Z}}\right\}.$$
As $\Rdnst$ is clearly an open subset of $\Rdn$ it carries the structure of an $dn$-dimensional manifold.
Furthermore, 
$\Rdn$ and $\Rdnst$ are both acted on from the right by the cyclic group of $n$ elements by cyclically permuting the indices. That is, if $j\in\frac{\mathbb{Z}}{n\mathbb{Z}}$ then
$$((\mathbf{x}_{1},\ldots, \mathbf{x}_{n}),j)\mapsto (\mathbf{x}_{1+j},\ldots, \mathbf{x}_{n+j})$$
where addition is modulo $n$.

\begin{remark}
In this remark we will observe that we can identify the above introduced space of discrete regular curves $\Rdnst$ with the set of piecewise linear, regular curves. Therefore we  let
$$\on{PL}^n(S^1,\R^d)=\left\{c\in C(S^1,\R^d): c|_{\left[\frac{i}{n},\frac{i+1}{n}\right]}\text{ is linear}\right\}$$
be the space of  continuous, piece-wise linear curves with $n$ control points from $S^1\to\R^d$
and 
$$\on{PLImm}^n(S^1,\R^d)=\left\{c\in \on{PL}^n(S^1,\R^d): c\left(\frac{i}{n}\right)\ne c\left(\frac{i+1}{n}\right),\ \forall  i\in \frac{\mathbb{Z}}{n\mathbb{Z}}\right\}$$
the open subset of piecewise linear immersions.
To identify the space of piecewise linear immersion with the previously defined space
$\Rdnst$, thereby allowing us to visualize $\Rdnst$ as a space of curves, we simply consider the map $$c\mapsto\left(c\left(\frac{0}{n}\right),c\left(\frac{1}{n}\right),\ldots,c\left(\frac{n-1}{n}\right)\right)\in\Rdnst$$
  Likewise we may identify $T_c\on{PLImm}^n(S^1,\R^d)\cong \on{PL}^n(S^1,\R^d)$ with $T_c\Rdnst\cong\Rdn$. 
\end{remark}

We now wish to define a Riemannian metric $\gsi^m$ on $\Rdnst$, which can be interpreted as a discretization of $\sgsi^m$. First, we define some notation.
Given a curve $c\in\Rdnst$ we denote our vertices as $c_i$ and let $e_i=c_{i+1}-c_i$ denote the edge beginning at $c_i$ and ending at $c_{i+1}$. We also denote the average length of the two edges meeting at the vertex $c_i$ as $\mu_i=\frac{|e_i|+|e_{i-1}|}{2}$. Let $h\in T_c\Rdnst$ denote a tangent vector to $c$ with $h_i$ denoting the $i$-th component of $h$. To define our discrete metric $\gsi^m$ we will also need to define the $j$-th discrete derivative at the $i$th vertex $D_s^j h_i$ using some ideas from discrete differential geometry \cite{crane2018discrete,crane2013digital}.

Given a curve $c\in \Rdnst$ and a function $h: \R^n\to \R^d$ via $c_i\mapsto h_i\in \R^d$ we want to give to edge $e_i$ a quantity $D_s h_i$ measuring the change from $c_i$ to $c_{i+1}$. Following the classic textbook by Crane~\cite{crane2018discrete} it is natural to set
$$D_s h_i:= (\star dh)_i,$$
where $\star$ is the discrete Hodge star operator. This maps an $k$-form on the primal curve to an $n-k$ form on the dual curve (or a $k$-form on the dual curve to an $n-k$ form on the primal curve). This does mean that our first derivatives will actually be associated with the duals of the primal edges rather than the primal edges themselves, but this is rather unimportant as they are in 1-1 correspondence to each-other.

This definition has a natural generalization to higher derivatives in $D_s^j h_i:=((\star d)^j h)_i$. For even $j$ these are 0-forms associated with vertices and for odd $j$ they are 0-forms of the dual curve associated with duals of edges. Note that differing from~\cite{crane2018discrete}, which treats general meshes, here we only have vertices and edges. Thus our dual ``mesh'' is simply a dual curve and consequently also only has dual vertices and dual edges, which are the edges and  vertices of the dual curve. In addition the dual curve is again oriented, albeit in the opposite direction compared to the primal curve

To make this definition more concrete we  expand the first and second derivative, which then reads as
\begin{align*}
  D_s h_i&:=(\star dh)_i=\frac{|e_i^*|}{|e_i|}(h_{i+1}-h_i)=\frac{1}{|e_i|}(h_{i+1}-h_i)\\
  \text{and}\\
D_s^2h_i&:=(\star d\star dh)_i=\frac{|v_i|}{|v_i^*|}(D_sh_{i}-D_sh_{i-1})=\frac{1}{\mu_i}(D_sh_{i}-D_sh_{i-1}).\end{align*}
In a similar manner one can show that this leads to the following definition for higher derivatives:
$$D_s^0 h_i:=h_i\qquad D_s^{j+1} h_i:=\begin{cases}\frac{D_s^{j}h_i-D_s^{j}h_{i-1}}{\mu_i}\quad j\notin2\mathbb{N}\\ \frac{D_s^{j}h_{i+1}-D_s^{j}h_i}{|e_i|}\quad j\in2\mathbb{N}\end{cases}.$$
A visualization of this construction can be seen in Figure~\ref{fig:discretederivpic}. This allows us to define the discrete analogous to the $m$-th order component of $\sgsi^m$.

\begin{figure}
    \centering
    \includegraphics[width=.90\textwidth]{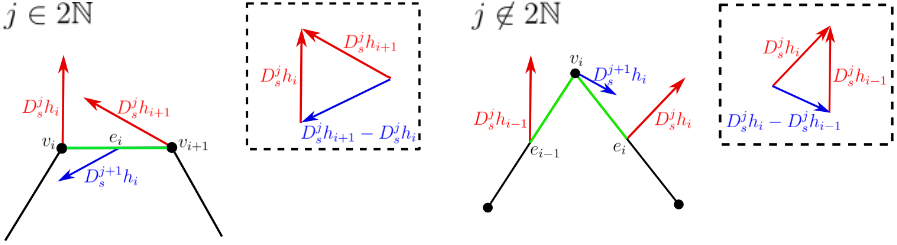}
    \caption{A pictorial representation of the discrete derivative. The derivative $D_s^j h$ can be seen here as a scaled first difference of $D_s^{j-1}h$. Both the even case and the odd case are shown. Note on the right for odd $j$ the derivative $D^{j+1}_sh_i$ is associated with the vertex $v_i$ while for even $j$ the derivative $D^{j+1}_sh_i$ is associated with the edge $e_i$.}
    \label{fig:discretederivpic}
\end{figure}

\begin{definition}
  For $c\in \Rdnst$ and $h,k\in \Rdn$ we let 
  $$\dot{\gsi}_c^m(h,k):=\sum_{i=1}^n\ell(c)^{2m-3}\langle D_s^m h_i,D_s^m k_i\rangle \mu_{i,m}\text{ where }\mu_{i,m}=\begin{cases}\mu_i\quad m\in 2\mathbb{Z}^+\\ |e_i|\quad m\notin 2\mathbb{Z}^+\end{cases}$$
  and 
    $$\gsi_c^m(h,k)=\dot{\gsi}_c^0(h,k)+\dot{\gsi}_c^m(h,k).$$
  \end{definition}
  \begin{remark}[An alternate definition of discrete derivatives]
    An alternative  definition of (higher order) discrete derivatives is given by
    $$D_s^0 h_i:=h_i\qquad D_s^{j+1} h_i:=\frac{D_s^{j}h_i-D_s^{j}h_{i-1}}{|e_i|}\quad \forall i\ge 1.$$
    At a first glance this definition seems more natural and significantly simpler. This more naive approach performs, however, significantly worse in numerical experiments as it can lead to several singularities and ill-behaviors of the corresponding discrete minimizers.  This served as the main motivation for adapting the methods of discrete exterior calculus to define the Riemannian metrics, and our numerical experiments confirmed that such a discretization behaves numerically significantly more robust. 
    
    We want to emphasize, that our main results (eg.~Theorem~\ref{mainbit}) are unaffected of this choice of discretization technique, i.e., our results still hold with this definition of discrete derivative.
  \end{remark}

  In the following lemma we further justify the above notation and show that it defines a Riemannian metric on $\Rdnst$.
\begin{lemma}\label{GmMetricDef}
  Let $n\ge 3$ and $m\ge 0$. Then
  $\gsi^m$
   is a Riemannian metric on $\Rdnst$.
\end{lemma}
\begin{proof}
  We need to check non-degeneracy of the inner product and the smooth dependence on the base point $c$.
  As $g_c^m(h,h)\ge\dot{g}_c^0(h,h)$ it follows that if $g_c^m(h,h)=0$ then we also have $\dot{g}_c^0(h,h)=0$. As each edge length $|e_i|$ and $\mu_i$ are nonzero for any $c$, this can only happen if $\langle D_s^0 h_i, D_s^0 h_i\rangle=\langle h_i, h_i\rangle=0$ for all $i$. That implies $h_i=0$ for all $i$. Thus for any $m$, $g_c^m(h,h)=0$ implies $h\equiv 0$.
  
  We next check that $\gsi^m_c$ varies smoothly with respect to the base point $c$. The only suspect term in the definition is $|e_i|$. Varying $c_i$ in the $u$ direction yields the following:
  $$\frac{\partial}{\partial u}|e_i|=\frac{-u\cdot e_i}{|e_{i}|}\qquad \frac{\partial}{\partial u}|e_{i-1}|=\frac{u\cdot e_{i-1}}{|e_{i-1}|}$$
  which implies the smoothness of the metric as $|e_i|>0$ for each~$i$ of any curve in $\Rdnst$.
\end{proof}
In the next Lemma we will show that higher order metrics dominate lower order metrics, which will be of importance in the next section, where we will establish the main completeness results of the present article.
\begin{lemma}\label{metricequivalence}
  Let $m\ge 1$,  $c\in \Rdnst$ and $h\in \Rdn$, then
  \begin{equation}
    \dot{\gsi}^m_c(h,h)\le \frac{1}{4}\dot{\gsi}^{m+1}_c(h,h).
  \end{equation}
\end{lemma}
\begin{proof}
The proof of this result is rather technical and we postpone it to the appendix. 
\end{proof}
With these lemmas in hand, we arrive at a result concerning the main properties of the discrete Riemannian metric $\gsi^m$, which parallels the statements of Lemma \ref{smoothproperties}.\\

\begin{lemma}\label{discreteinvariances}
  The metric $\gsi^m_c(h,k)$ has the following properties:
  \begin{enumerate}
  \item\label{item1} For all $n\ge 3$, $m\ge0$, $\gsi^m$ is invariant with respect to the action of $\frac{\mathbb{Z}}{n\mathbb{Z}}$ on $\Rdnst$, as well as rescaling, rotation, and translation. If $j\in\frac{\mathbb{Z}}{n\mathbb{Z}}$, $\lambda\in\R^+$, $R\in \on{SO}(\R^d)$ and $\mathbf{v}\in\R^d$, then
    \begin{align*} 
      \gsi_c^m(h,k)&=\gsi_{c\circ j}^m(h\circ j,k\circ j)=\gsi_{\lambda c}^m(\lambda h,\lambda k)=\gsi_{\mathbf{v}+c}^m(h,k)=\gsi_{Rc}^m(Rh,Rk).
    \end{align*}
  \item\label{item2} The metric $\gsi^m(h,k)$ is equivalent to any metric of the form
    $$\widetilde{\gsi}^m(h,k):=\sum_{j=0}^m \sum_{i=1}^na_j\ell(c)^{2j-3}\langle D_s^j h_i,D_s^j k_i\rangle\mu_{i,m}$$ for $a_0,a_m>0$ and $a_j\ge 0$ for $j=1,\ldots,m-1$ in the following sense: for some $C\in\R$, all $c\in \Rdnst$ and all $h\in T_c\Rdnst\cong\Rdn$
      $$\frac{1}{C}{\gsi}^m_c(h,h)\le \widetilde{\gsi}^m_c(h,h)\le C {\gsi}^m_c(h,h).$$
  \end{enumerate}
\end{lemma}
\begin{proof}
  The proof of statement \ref{item1} is similar to the smooth case and we will not present the proof for translation, rotation, or scale invariance. The key difference is the action by $\frac{\mathbb{Z}}{n\mathbb{Z}}$. Note $\gsi_c^m(h,k)$ and $\gsi_{c\circ j}^m(h\circ j,k\circ j)$ sum over the exact same terms. Therefore, since $|e_i\circ j|=|e_{i+j}|$ and $\mu_i\circ j=\mu_{i+j}$, we see
  $$\dot{\gsi}_{c\circ j}^m(h\circ j,k\circ j)=\sum_{i=1+j}^{n+j}\ell(c)^{2m-3}\langle D_s^m h_i,D_s^mk_i\rangle\mu_{i,m}=\sum_{i=1}^n\ell(c)^{2m-3}\langle D_s^m h_i,D_s^mk_i\rangle\mu_{i,m}=\dot{\gsi}_{c}^m(h,k)$$
  where the center equality is due to the fact that $i$ belongs to $\frac{\mathbb{Z}}{n\mathbb{Z}}$ so that the $i=1$ term of the sum is identical to the $i=n+1$ term.
  
  Statement \ref{item2} largely follows via Lemma \ref{metricequivalence}. Given our $\widetilde{\gsi}^m$, it allows us to collect all our coefficients into $\dot\gsi^0$ and $\dot\gsi^m$ terms alone to form a $\widehat{\gsi}^m$ with
  $$a_0\dot{\gsi}^0_c(h,h)+a_m\dot{\gsi}^m_c(h,h)\le \widetilde{\gsi}^m_c(h,h)\le\widehat{\gsi}^m_c(h,h).$$
  As $\widehat{\gsi}^m_c(h,h)=\hat{a}_0\dot{\gsi}^0(h,h)+\hat{a}_m\dot{\gsi}^m(h,h)$ for some $a_0\le \hat{a}_0$ and $a_m\le \hat{a}_m$. Note also that
  $$\min(\hat{a}_0,\hat{a}_m)\gsi^m_c(h,h)\le \widehat{\gsi}^m_c(h,h)\le \max(\hat{a}_0,\hat{a}_m)\gsi^m_c(h,h)$$
  and similar for $a_0\dot{\gsi}^0_c(h,h)+a_m\dot{\gsi}^m_c(h,h)$. Thus we may write
  $$\frac{1}{C}\gsi^m_c(h,h)\le\min(a_0,a_m)\gsi^m_c(h,h)\le \widetilde{\gsi}^m_c(h,h)\le \max(\hat{a}_0,\hat{a}_m)\gsi^m_c(h,h)\le C\gsi^m_c(h,h)$$
  for
  $$C=\max\left(\frac{1}{a_0},\frac{1}{a_m}, \hat{a}_0,\hat{a}_m\right).$$
\end{proof}
We have seen that this discrete metric shares several properties with its infinite dimensional counterpart. In the following proposition we show that we are indeed justified in referring to it as a discretization of $\sgsi^m$:\\

\begin{proposition}\label{convergenceproposition}
  Let $n\ge 3$ and for $i=0,\ldots, n-1$ let $\theta_i=\frac{i}{n}$. Let $c\in\Imm(S^1,\R^d)$ and $h,k\in T_c\Imm(S^1,\R^d)$. Define $\tilde{c}_n=(c(\theta_1),c(\theta_2),\ldots,c(\theta_n))\in\Rdnst$. Similarly define $\tilde{h}_n$ and $\tilde{k}_n$ using $h,k$. Then
  $$\lim_{n\to\infty} \gsi^m_{\tilde{c}_n}(\tilde{h}_n,\tilde{k}_n)=\sgsi^m_c(h,k)$$
\end{proposition}
\begin{proof}
We first define certain approximation operators, which we will need throughout the proof. Namely we consider
  $I^1_n(f)$, $\widetilde{I^1_n}(f)$, and $I_n^0(f)$ as piece-wise constant approximates of $f'$, $f'$, and $f$ respectively. Each is constant on each $\left[\theta_i,\theta_{i+1}\right)\subset [0,1]$ and all converge uniformly as $n\to \infty$. Both $I_n^1$ and $\widetilde{I_n^1}$ are needed due to the dependence on formulas on parity.
 We define them by prescribing their value on each interval, i.e., for $\theta$ in $[\theta_i,\theta_{i+1})$ we let
  $$I_n^0(f)(\theta)=f(\theta_i)\qquad I_n^1(f)(\theta)=n(f(\theta_{i+1})-f(\theta_i))\qquad \widetilde{I_n^1}(f)(\theta)=n(f(\theta_i)-f(\theta_{i-1})).$$
  The proof of their convergence is postponed to the appendix, cf.~Lemma \ref{derivativeapproximators}.
  
  Given $c\in\Imm(S^1,\R^d)$ we let $\widehat{\mu}_{c,n,m}=|I_n^1(c)|$ if $m$ is even and $\widehat{\mu}_{c,n,m}=\frac{|I_n^1(c)|+|\widetilde{I}_n^1(c)|}{2}$ if $m$ is odd. Now, for a given $c\in\Imm(S^1,\R^d)$ and $n\in\mathbb{Z}_{\ge 3}$ define $\mathcal{D}_{c,n,m}$ via the following recursive formula
  $$\mathcal{D}_{c,n,0}(f)=I^0_n(f)\qquad\mathcal{D}_{c,n,m}(f)=\begin{cases}\frac{\widetilde{I}^1_n(D_{c,n,m-1}(f))}{\widehat{\mu}_{c,n,m}}\quad m\in 2\mathbb{Z}^+\\\frac{I^1_n(D_{c,n,m-1}(f))}{\widehat{\mu}_{c,n,m}}\quad m\notin 2\mathbb{Z}^+\end{cases}.$$
  Note that all the $n$ cancel in these definitions once expanded as at each step we introduce a factor of $n$ in both the numerator and denominator.
  Next define $\ell_n(c)$ as $\sum_{i=1}^{n}|c(\theta_{i+1})-c(\theta_i)|$. With these in hand we define:
  $$F_m(h,k,c,n,\theta)=\ell_n(c)^{2m-3}\left\langle\mathcal{D}_{c,n,m}(h),\mathcal{D}_{c,n,m}(k)\right\rangle\widehat{\mu}_{c,n,m}.$$
  We prove the proposition in two steps
  \begin{claim}
    $\int_{S^1}F_m(h,k,c,n,\theta)d\theta=\dot{\gsi}^m_{\widetilde{c}_n}(\widetilde{h}_n,\widetilde{k}_n)$
  \end{claim}
  We expand to
  $$\int_{S^1}F_m(h,k,c,n,\theta)d\theta=\int_{S^1}\ell_n(c)^{2m-3}\left\langle\mathcal{D}_{c,n,m}(h),\mathcal{D}_{c,n,m}(k)\right\rangle\widehat{\mu}_{c,n,m}d\theta$$
  and note that the integrand on the right is constant on each interval $[\theta_i,\theta_{i+1})$ as $\ell_n$ is constant with respect to $\theta$ and each of $\mathcal{D}_{c,n,m}(h),\mathcal{D}_{c,n,m}(k)$, and ${\mu}_{c,n,m}$ are constant on such intervals. Thus for some $\widetilde{\mathcal{D}}_{c,n,m,i}(h),\widetilde{\mathcal{D}}_{c,n,m,i}(k)\in\R^d$ we can write either  $$\int_{S^1}F_m(h,k,c,n,\theta)d\theta=\sum_{i=1}^{n}\ell_n(c)^{2m-3}\frac{1}{n}\left\langle \widetilde{\mathcal{D}}_{c,n,m,i}(h),\widetilde{\mathcal{D}}_{c,n,m,i}(k)\right\rangle |n(c(\theta_{i+1})-c(\theta_i))|$$
  or
  $$=\sum_{i=1}^{n}\ell_n(c)^{2m-3}\frac{1}{n} \left\langle\widetilde{\mathcal{D}}_{c,n,m,i}(h),\widetilde{\mathcal{D}}_{c,n,m,i}(k)\right\rangle\frac{|n(c(\theta_{i+1})-c(\theta_i))|+|n(c(\theta_{i})-c(\theta_{i-1}))|}{2}$$
  depending on the parity of $m$. Note all $n$ cancel and we can rewrite further to either
  $$\sum_{i=1}^{n}\ell_n(c)^{2m-3}\left\langle \widetilde{\mathcal{D}}_{c,n,m,i}(h),\widetilde{\mathcal{D}}_{c,n,m,i}(k)\right\rangle|e_i|$$
  for odd $m$ or
  $$\sum_{i=1}^{n}\ell_n(c)^{2m-3}\left\langle \widetilde{\mathcal{D}}_{c,n,m,i}(h),\widetilde{\mathcal{D}}_{c,n,m,i}(k)\right\rangle\mu_i$$
  for even $m$. Thus the only thing that remains to be shown for this first claim is that these $\widetilde{\mathcal{D}}$ terms are our discrete $D_s$ derivatives of $\widetilde{h}_n$ and $\widetilde{k}_n$. This is clear for $m=0$ as $I^0_n(h)$ is simply the step function taking on the value $\widetilde{h}_{n,i}=h\left(\frac{i}{n}\right)$ on $[\theta_i,\theta_{i+1})$. For $m>0$ we see either
  $$\widetilde{\mathcal{D}}_{c,n,m,i}(h)=\frac{\widetilde{\mathcal{D}}_{c,n,m-1,i+1}(h)-\widetilde{\mathcal{D}}_{c,n,m-1,i}(h)}{|e_i|}$$
  or
  $$\widetilde{\mathcal{D}}_{c,n,m,i}(h)d\theta=\frac{\widetilde{\mathcal{D}}_{c,n,m-1,i}(h)-\widetilde{\mathcal{D}}_{c,n,m-1,i-1}(h)}{\mu_i}$$
  so that by induction it follows that these really do correspond to the discrete derivatives as $\widetilde{\mathcal{D}}_{c,n,m,i}(h)$ relates to $\widetilde{\mathcal{D}}_{c,n,m-1,j}(h)$ identically to how $D_s^mh_i$ relates to $D_s^{m-1} h_j$.\\
  The second step is easier by comparison. We have shown that $\int_{S^1}F_m(\dotsm)=\dot{\gsi}_{\tilde{c}_n}^m(\dotsm)$. We next need to show how the limit of $F_m$ is related to $\dot{\sgsi}^m_c(h,k)=\ell^{2m-3}\langle D_s^m h,D_s^m k\rangle|c'|$ which will allow us to relate $\gsi^m$ and $\sgsi^m$.
  \begin{claim}
    $\displaystyle{\lim_{n\to\infty} F_m(h,k,c,n,\theta)\overset{L^\infty}\to \dot{\sgsi}^m_c(h,k)}$   \end{claim}
  By Lemma \ref{derivativeapproximators} it follows that
  $$\lim_{n\to\infty}\left\|\mathcal{D}_{c,n,m}(f)-D_s^m(f)\right\|_{L^\infty}\to 0\quad\text{and}\quad\lim_{n\to\infty}\|\hat{\mu}_{c,n,m}-|c'|\|_{L^\infty}\to0.$$
  It is also clear that $\lim_{n\to\infty}\ell_n(c)\to \ell(c)$ as our curves are rectifiable.
  Now $$\lim_{n\to\infty}\|F_m(h,k,c,n,\theta)-\dot{\sgsi}^m_c(h,k)\|_{L^\infty}\to 0.$$
  This implies that $\displaystyle{\lim_{n\to\infty}F_m(\dotsm)}$ is equal to $\dot{\sgsi}^m_c(h,k)$ almost everywhere, and we may join our claims together via
  $$\int_{S^1}\dot{\sgsi}^m_c(h,k)d\theta=\int_{S^1} \lim_{n\to\infty} F_m(\dotsm) d\theta\overset{*}=\lim_{n\to\infty} \int_{S^1}F_m(\dotsm)d\theta=\lim_{n\to\infty}\dot{\gsi}^m_{\tilde{c}_n}(\tilde{h}_n,\tilde{k}_n),$$
  where $*$ follows by the Lebesgue dominated convergence theorem as we can bound the integrand almost everywhere by the constant function on $S^1$ equal to $\|\dot{\sgsi}^m_c(h,k)\|_{L^{\infty}}$. By linearity it follows
  $$\sgsi^m_c(h,k)=\lim_{n\to\infty}\gsi^m_{\tilde{c}_n}(\tilde{h}_n,\tilde{k}_n)$$
  which completes the proof.
\end{proof}

\section{Completeness Results}
In this section we will prove the main theoretical result of the present article, namely we will show that the finite dimensional manifolds $(\Rdnst,\gsi^m)$ indeed capture the completeness properties of the space of smooth, regular curves equipped with the class of reparametrization invariant Sobolev metrics $\sgsi^m$, i.e., we will prove the  following theorem:\\

\begin{thm}[Completeness of the $\gsi^m$-metric]\label{mainbit}
  Let $n\ge 3$ and $d\ge 2$. Then the space $(\Rdnst,\gsi^m)$ is metrically and geodesically complete if and only if $m\ge 2$. Furthermore,  for any two points in $\Rdnst$ there exists a minimizing geodesic, i.e.,  $(\Rdnst,\gsi^m)$  is geodesically convex.
\end{thm}
This theorem can readily be seen as a discrete analog to the above Theorem~\ref{smoothcomplete} from the smooth case. As $\Rdnst$ is finite dimensional, Hopf-Rinow implies we need only show metric completeness and that both geodesic completeness and geodesic convexity follow automatically. The requirement that $d\ge 2$ is due to the fact that $\R^{1\times n}_*$ is not connected and thus Hopf-Rinow does not apply. In the following Lemma we observe that metric incompleteness reduces to showing that there exists finite length paths that leave the space:
\begin{lemma}\label{pathstatement}
  Let $(\mathcal M,g)$ be a possibly infinite dimensional manifold. $(\mathcal M,g)$ is metrically incomplete if and only if there exists a path $\gamma:[0,1)\to \mathcal M$, such that 
  \begin{enumerate}
  \item $\gamma(t)\in \mathcal M$ for $0\leq t<1$;
  \item $\displaystyle{\lim_{t\to 1}}\, \gamma(t)$ does not exist in $(\mathcal M,g)$;
  \item the length of $\gamma$ (w.r.t. $g$) is finite, i.e., $\L_g(\gamma)<\infty$.
  \end{enumerate}
\end{lemma}
\begin{proof}
  We first assume incompleteness. As $\mathcal{M}$ is metrically incomplete there exists some Cauchy sequence $x_n$ that fails to converge in the space. Without loss of generality we can assume $\sum_n^\infty d_g(x_n,x_{n+1})<\infty$, since we can always find a Cauchy sub-sequence of our Cauchy sequence with this property. To construct such a sub-sequence we perform the following procedure: if, under our Cauchy conditions, for all $n,m\ge N(i)$ we have $d_g(x_n,x_m)<\frac{1}{10^i}$ then we define our sub-sequence as $\{x_{N(i)}\}_{i\in\mathbb{N}}\subset \{x_n\}$. The length $\sum_{i=1}^\infty d_g( x_{N(i)},x_{N(i+1)})$ of this sub-sequence is bounded above by $\frac{1}{9}$ by construction as $\sum_{i=1}^\infty d_g( x_{N(i)},x_{N(i+1)})<\sum_{i=1}^\infty \frac{1}{10^i}=\frac{1}{9}$.  
  While we cannot guarantee paths lying in the space between the elements of our sequence that realize their distances, we can have a collection of paths $\gamma_n$ going from $x_n$ to $x_{n+1}$ with a length $<d_g(x_n,x_{n+1})+\frac{1}{2^n}$ lying within the space so that the overall distance of $\gamma=\gamma_1*\gamma_2*\gamma_3*\dotsm$ will have a total length $L(\gamma)< 1+\sum_{n=1}^\infty d_g(x_n,x_{n+1})$. Here $*$ denotes path concatenation. $\gamma$ is then of finite length with $\lim_{t\to 1}\gamma(t)\notin \mathcal{M}$ and $\gamma(t)\in \mathcal{M}$ for $0\le t< 1$.  
 On the other hand, assuming such a path, we wish to show we have a Cauchy sequence that does not converge with respect to our metric. Without loss of generality, assume $\gamma$ is traversed with constant speed. Since $L_g(\gamma)<\infty$ the length of any subsegment of our curve is also finite.
  If $a$ and $b$ are points on our curve, then let $d(a,b)$ be the length of the subsegment of gamma with endpoints $\gamma(a),\gamma(b)$. Since $\gamma$ is traversed with constant speed, we see $d(t_i,t_j)= \L_g(\gamma)|t_i-t_j|$. The distance between $\gamma(t_i)$ and $\gamma(t_j)$ along $\gamma$ is not necessarily equal to the distance between these points in the space, but if $d_g$ is the metric distance, then $d_g(\gamma(t_i),\gamma(t_j))\le d(t_i,t_j)= \L_g(\gamma)|t_i-t_j|$. Set $t_n=1-\frac{1}{n}$ noting that this sequence $(t_n)$ is a Cauchy sequence in the reals. The distances between a pair of these points is $d_g(\gamma(t_i),\gamma(t_j))\le \L_g(\gamma)|t_i-t_j|$. Now given any $\varepsilon>0$ we have some $N$ so that for all $i,j>N$ we have $|t_i-t_j|<\frac{\varepsilon}{\L_g(\gamma)}$ so that $d_g(\gamma(t_i),\gamma(t_j))<\varepsilon$. This makes $\gamma(t_n)$  a Cauchy sequence in $\mathcal M$, but $(t_n)$ converges to 1 so this Cauchy sequence does not converge in our space:
  $$\lim_{n\to \infty}\gamma(t_n)=\lim_{t\to 1}\gamma(t)\notin\mathcal{M}.$$
  Thus our space is not metrically complete.
\end{proof}

Next we recall a useful  lemma, which will allow us to prove boundedness of certain quantities on metric balls and will be crucial for the desired completeness results. We state the lemma here in a much more narrow and simple context than in its original form:
\begin{lemma}[Lemma 3.2 in \cite{bruveris2015completeness}]\label{lem:usefulreference}
  Let $(\mathcal{M},g)$ be a Riemannian manifold with a weakly continuous metric and $f:\mathcal{M}\to \R$ a $C^1$-function. Assume that for each metric ball $B(y,r)$ in $\mathcal{M}$ there exists a constant $C$, such that
  $$\left|\frac{df}{dh}\right|\le C(1+|f(x)|)\|h\|_g$$
  holds for all $x\in B(y,r)$ and all $h\in T_x\mathcal{M}$. Then the function
  $$f:(M,d)\to (\R,|\cdot|)$$
  is continuous and Lipschitz continuous on every metric ball. In particular, $f$ is bounded on every metric ball.
\end{lemma}
We will use this result to prove boundedness of certain quantities in the following lemma. Note, that this result is the exact discrete analog of Lemma 4.4 of~\cite{bauer2024completeness}.
\begin{lemma}\label{BoundedOnMetricBalls} Let $m\ge2$ and let $B_{\gsi^m}(\hat{c}_0,r)\subset\Rdnst$ be a $g^m$-metric ball centered at $\hat{c}_0\in \Rdnst$. Then both total curve length $\ell(c)$, and edge length $|e_i|$ are bounded from above and from below away from zero on $B_{\gsi^m}(\hat{c}_0,r)$, i.e., there exist constants $C_1,C_2>0$ (depending only on $r,m$ and $\hat{c}_0$) such that for all $c\in B_{\gsi^m}(\hat{c}_0,r)$ and all $i\in\{1,\ldots,n-1\}$  we have
\begin{equation}
C_1 < \ell(c)<C_2,\text{ and } C_1 < |e_i|<C_2.
\end{equation}
\end{lemma}
\begin{proof}
    Using Lemma~\ref{lem:usefulreference} we will show that for each $i\in\{0,\ldots,n-1\}$ the function $\ln(|e_i|)$  is Lipschitz on any $g^m$-metric ball $B_{g^m}(\hat{c}_0,r)$ . This in turn implies that $|e_i|$ and $|e_i|^{-1}$ are each bounded above away from infinity on $B_{g^m}(\hat{c}_0,r)$ and, by extension, bounded below away from 0 on the same set. As $\displaystyle{\ell(c)=\sum_{i=0}^{n-1}|e_i|}$ this will also imply $\ell(c)$ is  bounded above away from infinity and below away from zero on $B_{g^m}(\hat{c}_0,r)$.\\
  We estimate:
    \begin{align*} &\left|\frac{d}{dh}\ln(|e_i|)\right|=\left| |e_i|^{-1}\frac{\langle h_{i+1}-h_i,e_i\rangle}{|e_i|}\right|\le \left|\frac{|h_{i+1}-h_i|}{|e_i|}\right|\overset{(*_1)}\le\frac{1}{2}\sum_{i=1}^n\left|\frac{h_{i+1}-h_i}{|e_i|}-\frac{h_i-h_{i-1}}{|e_{i-1}|}\right|\\ &\qquad=\frac{1}{2}\sum_{i=1}^n\left|\frac{h_{i+1}-h_i}{|e_i|}-\frac{h_i-h_{i-1}}{|e_{i-1}|}\right|\frac{\sqrt{\mu_i}}{\sqrt{\mu_i}}\le \frac{1}{2}\left(\sum_{i=1}^n\left|\frac{h_{i+1}-h_i}{|e_i|}-\frac{h_i-h_{i-1}}{|e_{i-1}|}\right|^2\frac{1}{\mu_i}\right)^\frac{1}{2}\left(\sum_{i=1}^n \mu_i\right)^\frac{1}{2}\\
      &\qquad=\frac{1}{2}\left(\sum_{i=1}^n\ell(c)\left|\frac{h_{i+1}-h_i}{|e_i|}-\frac{h_i-h_{i-1}}{|e_{i-1}|}\right|^2\frac{1}{\mu_i}\right)^\frac{1}{2}= \frac{1}{2}\sqrt{\dot{\gsi}^2(h,h)}\overset{(*_2)}\le \frac{1}{2^{m-1}}\sqrt{\dot{\gsi}^m(h,h)}\le \frac{1}{2^{m-1}}\|h\|_{g^m}.
    \end{align*}
    where $(*_1)$ follows from Lemma \ref{clm:intermediatetwo} found in Appendix C and $(*_2)$ follows via Lemma~\ref{metricequivalence}. This proves Lemma~\ref{BoundedOnMetricBalls}.
  \end{proof}
We are now ready to prove Theorem~\ref{mainbit}.
\begin{proof}[Proof of Theorem~\ref{mainbit}]
  We will first show the incompleteness of $(\Rdnst,\gsi^m)$ for $m=1$; the case $m=0$ follows from the same calculation.
  Therefore, we consider the path $\gamma:[0,1)\to \Rdnst$ via
$$\gamma_{i}\left(t\right)=\left(i,0,\ldots,0\right)\ 0\le i\le n-2\qquad \gamma_{n-1}\left(t\right)=\left(0,1-t,\ldots\right)$$
Note $\lim_{t\to 1}\gamma(t)$ is not an element of $\Rdnst$ so that $(\Rdnst, \gsi^1)$ will be incomplete if $L_{\gsi^1}[\gamma]<\infty$. This is since if $\gamma$ is of finite length then $\gamma(1-\frac{1}{i})$ is a Cacuhy sequence in $\gsi^1$ such that $\lim_{i\to\infty}\gamma(1-\frac{1}{i})=\lim_{t\to1}\gamma(t)\notin\Rdnst$.

We find
\begin{align*}  L_{g^1}[\gamma]&=\int_0^1\sqrt{g^1(D_t\gamma,D_t\gamma)}dt=\int_0^1\sqrt{\sum_{i=0}^{n-1}\left( \frac{|D_t\gamma_i(t)|^2\mu_i}{\ell(\gamma(t))^3}+\frac{|D_t e_i|^2}{\ell(\gamma(t))|e_i|}\right)}dt\\
                               &=\int_0^1\sqrt{\left( \frac{\mu_{n-1}}{\ell(\gamma(t))^3}+\frac{1}{\ell(\gamma(t))}\left(\frac{1}{|e_0|}+\frac{1}{|e_{n-1}|}\right)\right)}dt\\
                               &=\int_{0}^{1}\left(\frac{\sqrt{(n-2)^2+(1-t)^{2}}+1-t}{2\left(\sqrt{(n-2)^2+(1-t)^{2}}+n-t-1\right)^{3}}\right.\\
                               &\qquad\left.+\frac{1}{\left(\sqrt{(n-2)^2+(1-t)^{2}}+n-t-1\right)}\left(\frac{1}{1-t}+\frac{1}{\sqrt{(n-2)^2+(1-t)^{2}}}\right)\right)^\frac{1}{2}dt\\
  &\le \int_0^1\sqrt{\frac{n-2+2(1-t)}{1}+\frac{1}{1}(\frac{1}{1}+\frac{1}{1})}dt=\int_0^1\sqrt{n-t+2}dt<\sqrt{n+2}<\infty.
\end{align*}

Next we prove the metric completeness for $m\ge 2$. 
  Therefore we first show  that for $m\geq 2$ the $g^m$-geodesic distance dominates the Euclidean distance on $g^m$-metric balls, i.e., for each metric ball $B_{g^m}(\hat{c}_0,r)$ there exists some constant $C$ dependent only on $\hat{c}_0$ and $r$ such that
    \begin{equation}\label{eq:metricdom}
        \on{dist}_{\Rdn}(\hat{c}_1,\hat{c}_2)\le C \on{dist}_{g^m}(\hat{c}_1,\hat{c}_2)
    \end{equation}
for all $\hat{c}_1,\hat{c}_2\in B_{g^m}(\hat{c}_0,r).$ We estimate using only the $\dot{\gsi}^0$-part to obtain:
  \begin{align*}
    \L_{\gsi^m}(\pathofcurves(t))&\ge\int_0^1\sqrt{\sum_{i=1}^n\frac{1}{\ell(\pathofcurves(t))^3}|D_t\pathofcurves(t)_i|^2\mu_i}dt\ge\int_0^1\frac{1}{\ell(\pathofcurves(t))^{3/2}}|D_t \pathofcurves(t)_j|\sqrt{\mu_j}dt\\
    &\ge M_j\int_0^1|D_t \pathofcurves(t)_j|dt=M_j\L_{\R^d}(\pathofcurves(t)_j)\quad \text{ for each } j
  \end{align*}
  Note by Lemma~\ref{BoundedOnMetricBalls}, $M_j=\inf_{t} \sqrt{\frac{\mu_j}{\ell(\gamma(t))^3}}>0$ for any path in our metric ball since $m\ge 2$. Now by averaging we see
    $$\L_{\gsi^m}(\pathofcurves(t))\ge \sum_{j=1}^n\frac{M_j}{n}\L_{\R^d}(\gamma(t)_j) \ge C\L_{\Rdn}(\pathofcurves(t))$$
    for some constant $C$. By taking infimums over paths in $B_{g^m}(\hat{c}_0,r)$ between $\hat{c}_1$ and $\hat{c}_2$ we have thus shown~\eqref{eq:metricdom}.\\
    Finally, we may show $(\Rdnst,\gsi^m)$ is complete for $m\ge 2$. 
  Therefore we consider a $\gsi^m$-Cauchy sequence $(\hat{c}_i)$ in $\Rdnst$. This must lie entirely within some $\gsi^m$ metric ball $B_{\gsi^m}(\hat{c}_0,r)$. Now ~\eqref{eq:metricdom} implies this is also a Cauchy sequence under the $\Rdn$ distance function. Using that $(\Rdn,\|\cdot\|_{\Rdn})$ is complete there exists $\hat{c}_\infty\in \Rdn$ such that  $$\on{dist}_{\Rdn}(\hat{c}_i,\hat{c}_\infty)=\|\hat{c}_i-\hat{c}_\infty\|_{\Rdn}
    \to 0.\qquad (*)$$
    In particular, as Lemma~\ref{BoundedOnMetricBalls} tells us edge lengths are bounded below away from zero on $B_{\gsi^m}(\hat{c}_0,\hat{c}_\infty)$ for any $m\ge 2$ it follows they are bounded below away from zero on the elements of the sequence $\hat{c}_i$ if $m\ge 2$. This means $\hat{c}_\infty$ must have nonzero edge lengths and thus $\hat{c}_\infty\in\Rdnst$. Now
    $\on{dist}_{\gsi^m}(\hat{c}_i,\hat{c}_\infty)\to 0$ so that $(\hat{c}_i)\to \hat{c}_\infty$ in $(\Rdnst,\gsi^m)$. Thus $\gsi^m$-Cauchy sequences in $(\Rdnst,\gsi^m)$ must have limits in $(\Rdnst,\gsi^m)$, i.e., $(\Rdnst,\gsi^m)$ is metrically complete.
\end{proof}

\section{Geodesics, curvature and a comparison to Kendall's shape space}\label{sec:numerics}
\begin{wrapfigure}[12]{r}{0.32\textwidth}
\vspace{-1.1cm}
  \begin{center}
    \includegraphics[width=.9\linewidth]{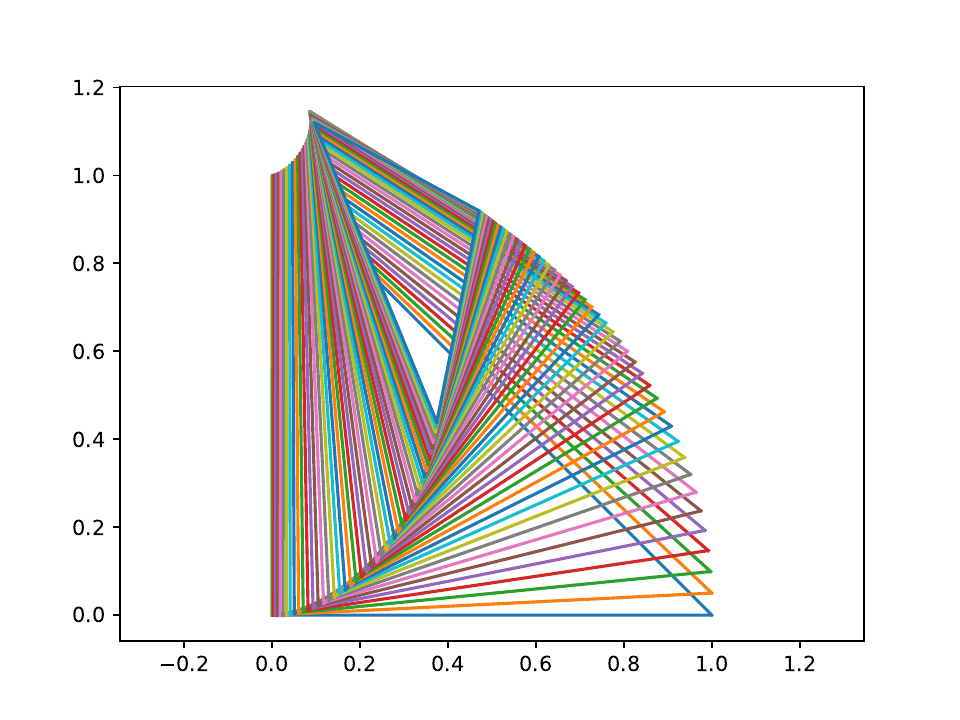}
    \end{center}
  \caption{Riemannian exponential map w.r.t. to the metric $\gsi^2$.}
  \label{fig:sc11}
\end{wrapfigure}
In this section we will further demonstrate the behavior of these geometries by  presenting selected examples of geodesics for different choices of metrics. Finally, we will consider the special case of planar triangles, where we will compare the discrete Sobolev metrics defined in this paper with Kendall's shape metric; it is well known that under Kendall's metric the space of planar triangles reduces to a round sphere~\cite{kendall1984shape}.

All numerical examples were obtained using the programming language \texttt{Python}, where we implemented both the Riemannian exponential map (and the Riemannian log map) by solving the the geodesic initial value problem (and boundary value problem, resp.).

\paragraph{The geodesic initial value problem:}To approximate the Riemannian exponential map we calculate the Christoffel symbols using the automatic differentiation capabilities of \texttt{pytorch}. This in turn allows us to approximate the exponential map using a simple one-step Euler method. In Figure \ref{fig:sc11} one can see an approximated geodesic computed with initial conditions consisting of a right triangle $(0,0),(0,1),(1,0)$ and an initial velocity of $(0,1)$ on the third vertex and zero elsewhere. We note that an initial velocity which is non-zero only on a single vertex immediately puts multiple vertices in motion. 

\paragraph{The geodesic boundary value problem:}To approximate the Riemannian log map we employ a path-straigthening algorithm; i.e.,  we minimize the Riemannian energy over paths of curves starting at $c_0$ and ending at $c_1$, approximated using a fixed (finite) number of intermediate curves. Thereby, we reduce the approximation of the Riemannian log map to a finite dimensional, unconstrained minimization problem, which we tackle  using the \texttt{scipy} implementation of L-BFGS-B, where we use a simple linear interpolation as initialization. In particular, for higher order metrics the algorithm can fail if the initialization leaves the space of PL immersions, as the metric is undefined for curves that are not in $\Rdnst$. To fix this issue it is possible to add a small amount of  noise to the vertices of the initialization. We present two different examples of solutions to the geodesic boundary value problem in Figure~\ref{fig:BVPexamples}.
\begin{figure}[h]
  \begin{center}
  \includegraphics[trim={2in 0 1.6in 0},clip,width=\linewidth]{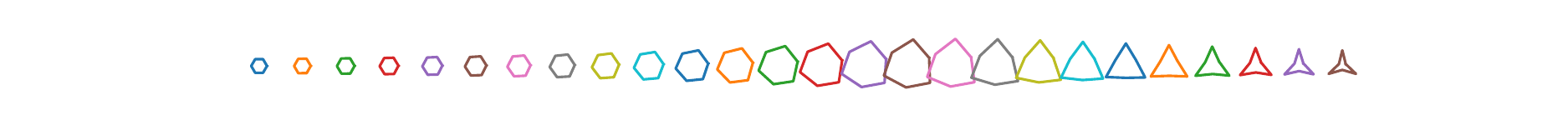}
  \includegraphics[trim={2in 0 1.6in 0},clip,width=\linewidth]{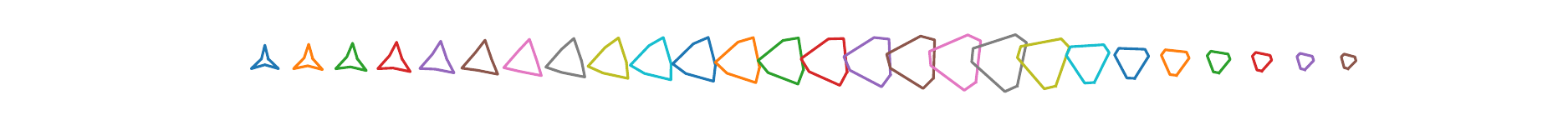}
  \end{center}
  \caption{Examples of solutions to the geodesic boundary value problem w.r.t. to the metric $\gsi^2$. First line:  from a hexagon to a spiked figure. Second line:  from the spiked figure to a blunt one.}
  \label{fig:BVPexamples}
\end{figure}

\paragraph{Gaussian Curvature of the Space of Triangles and Kendall's shape metric:}
A prudent comparison of the discrete metrics described in this work is with the Kendall metrics of discrete shapes, which may be defined on the same space. Comparing the geodesics with respect to these metrics also  highlights the nature of our completeness result. 

We start by restricting ourselves to the space of triangles $\mathbb{R}^{2\times 3}_*$ modulo the actions of rotation, translation, and scale. This space can be identified with the surface of a sphere with three punctures (the punctures correspond to degenerate triangles in which two vertices coincide) and when equipped with the Kendall metric \cite{kendall1984shape}, this space is famously isometric to the punctured sphere and has constant Gaussian curvature. In Figure \ref{fig:geocurves2}, we display boundary value geodesics with respect to the Kendall metric as well as for the discrete Sobolev metrics for $m=0,1,2$. While the geodesic with respect to the Kendall metric passes directly through a discrete curve where adjacent points coincide, each of the geodesics with respect to the metrics proposed in this paper does not pass through this point in the space. Moreover, geodesics with respect to the higher-order metrics pass further from this point.
\begin{figure}[h]
  \centering
\includegraphics[width=.14\textwidth]{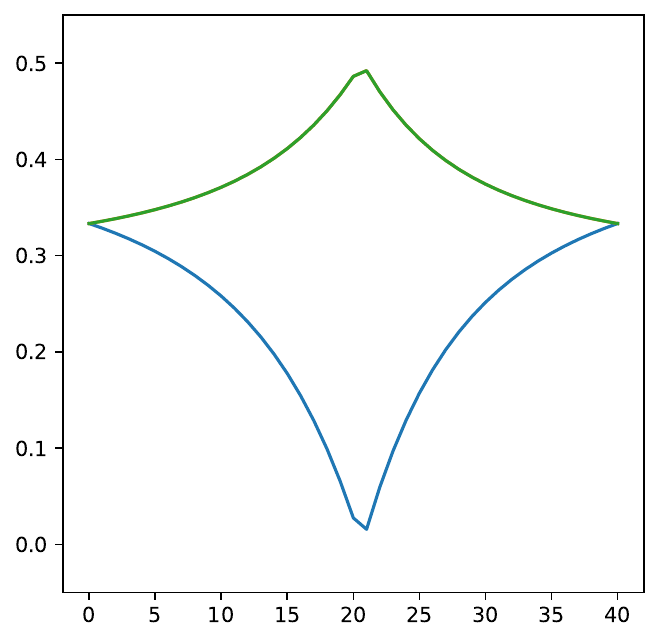}
  \includegraphics[width=.8\textwidth]{Geod_AL0.pdf}
\includegraphics[width=.14\textwidth]{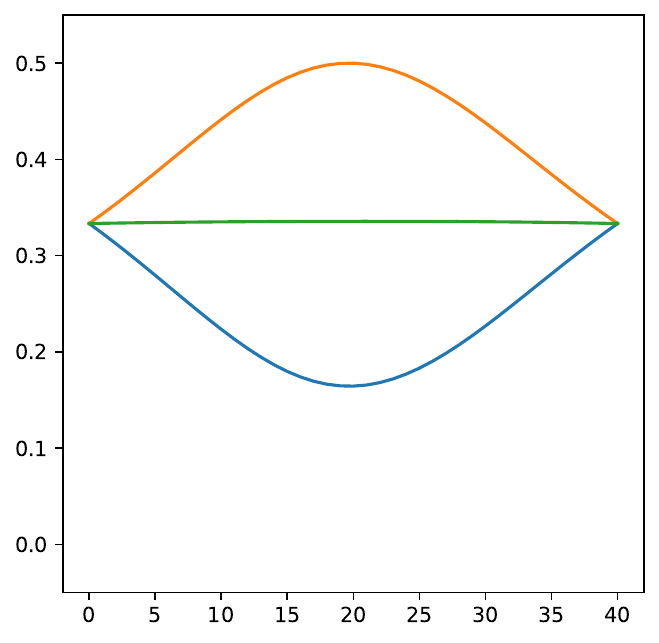}
  \includegraphics[width=.8\textwidth]{Geod_AL1.pdf}
\includegraphics[width=.14\textwidth]{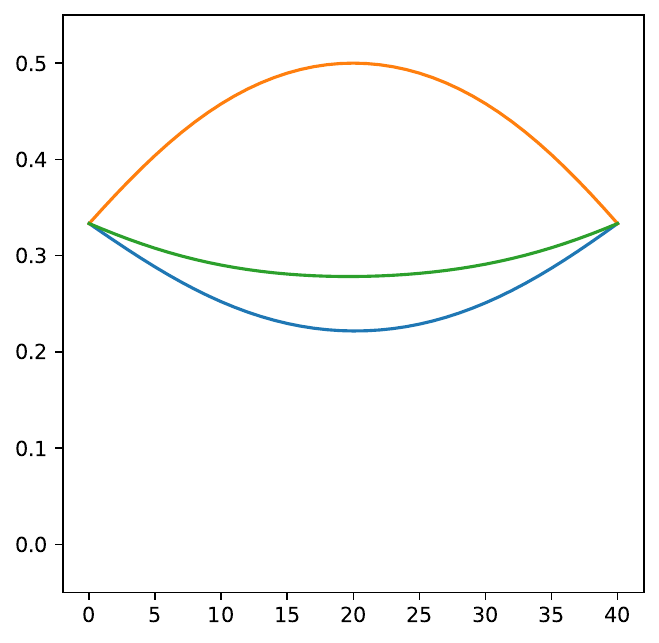}
  \includegraphics[width=.8\textwidth]{Geod_AL2.pdf}
  \includegraphics[width=.14\textwidth]{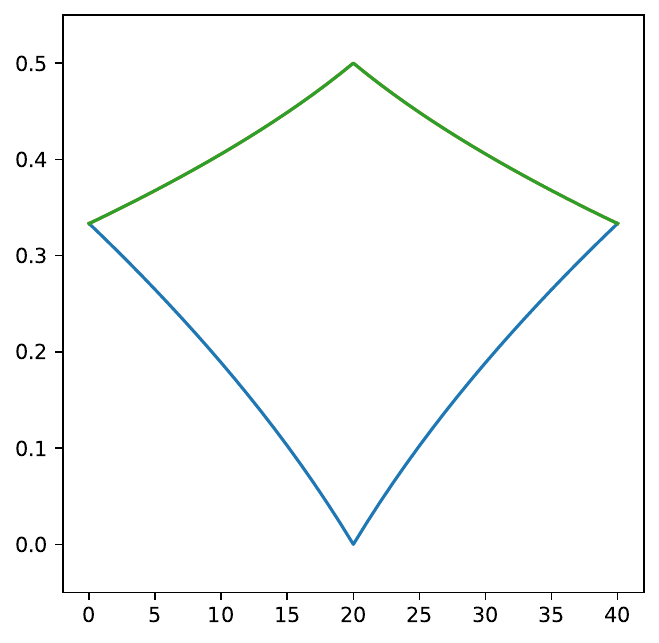}
\includegraphics[width=.8\textwidth]{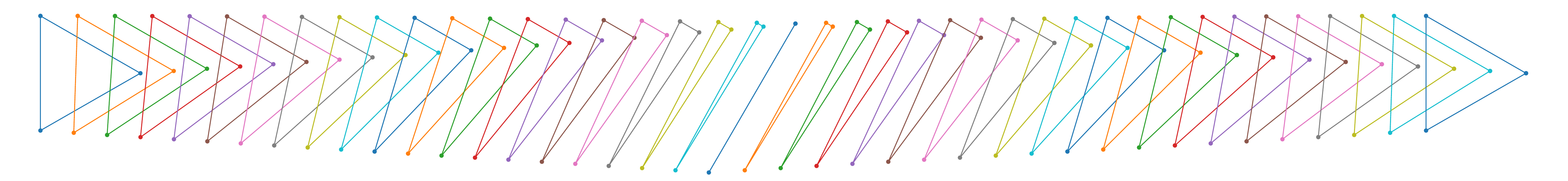}
  \caption{Geodesics in the space of triangles w.r.t. to four different Riemannian metrics. We display geodesics with respect to $\gsi^0$ (first line), $\gsi^1$ (second line), $\gsi^2$ (third line), and the Kendall metric (last line). To the left of each is a plot of the three edge lengths of the triangles along each geodesic.}
  \label{fig:geocurves2}
\end{figure}

Finally, we calculated numerical approximations of Gaussian curvature for the space of triangles w.r.t to the same four Riemannian metrics: while the Gaussian curvature for the Kendall metric is constant, this is clearly not the case for discrete Sobolev metrics of the present article, cf. Figure \ref{fig:spheres}.
Indeed all of the $\gsi^m$ metrics exhibit positive curvature near the points corresponding to triangles with double-points (which do not correspond to elements of $\Rdnst$), but as we approach the singularities (c.f. Figures 5 and 6), for $m=0$ the curvature decreases towards 0, for $m=1$ the curvature approaches a positive finite value, and for $m=2$ the curvature approaches positive infinity. More generally, an increase in order leads for a more curved space both for negatively curved regions but also for positively curved regions. Finally, in Figure \ref{fig:geocurves1}, we further visualize the curvature along selected paths to further demonstrate the behavior at key points of interest. One can see again the increase in positive curvature near the punctures of the sphere. 
\begin{figure}[H]
  \centering
  \includegraphics[width=.7\textwidth]{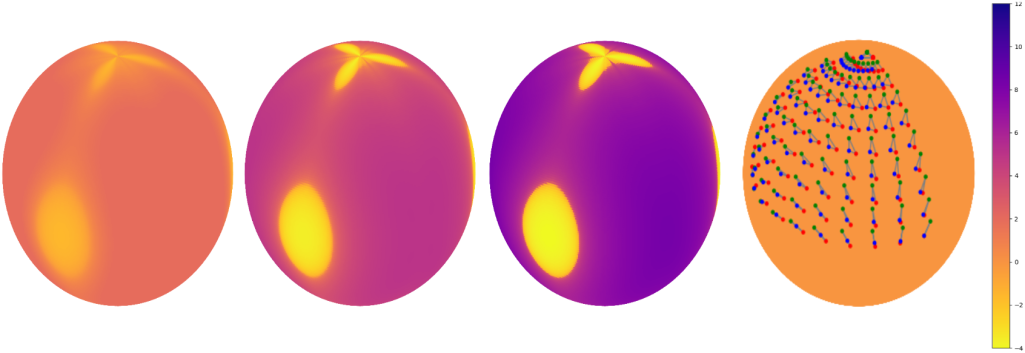}
  \caption{The Gaussian curvatures of the space of triangles triangles under the Kendall metric and the discretized scale invariant Sobolev metric $\gsi^m$ for $m=0,1,2$). From left to right we have $\gsi^m$ for $m=0,1,$ and $2$ and the Kendall metric (with the triangles drawn in). The scale is a symmetric log scale given by $\on{sign}(x)\log(|x+\on{sign}(x)|)$.}
  \label{fig:spheres}
\end{figure}

\begin{figure}[h]
  \centering
  \includegraphics[width=.9\textwidth]{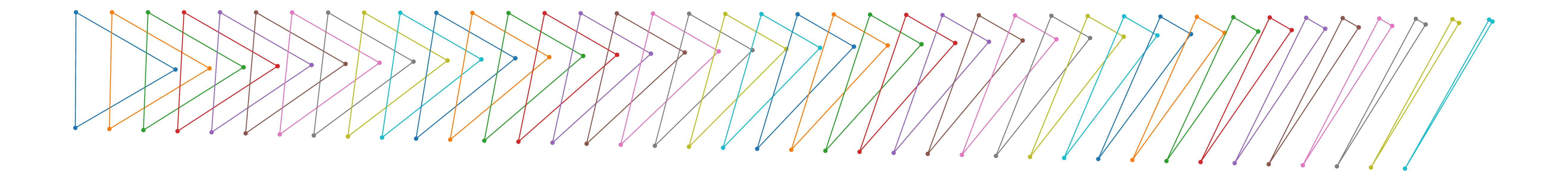}
  \includegraphics[width=.9\textwidth]{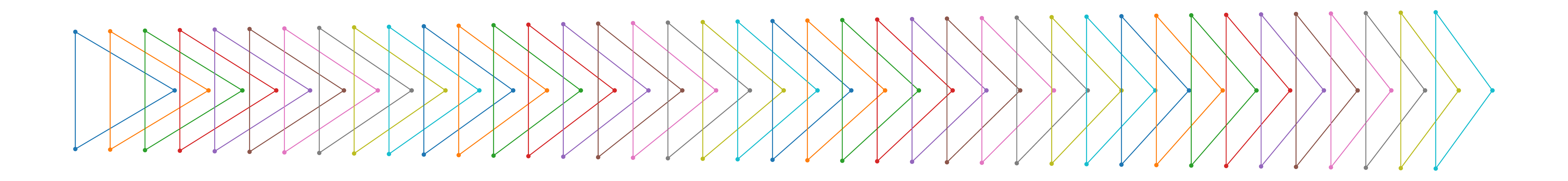}  \includegraphics[width=.45\textwidth]{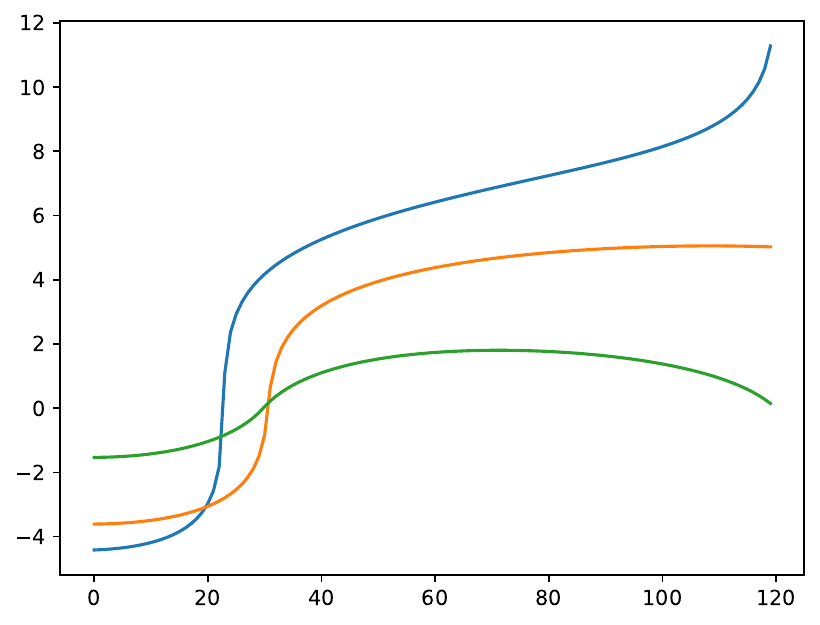}  \includegraphics[width=.45\textwidth]{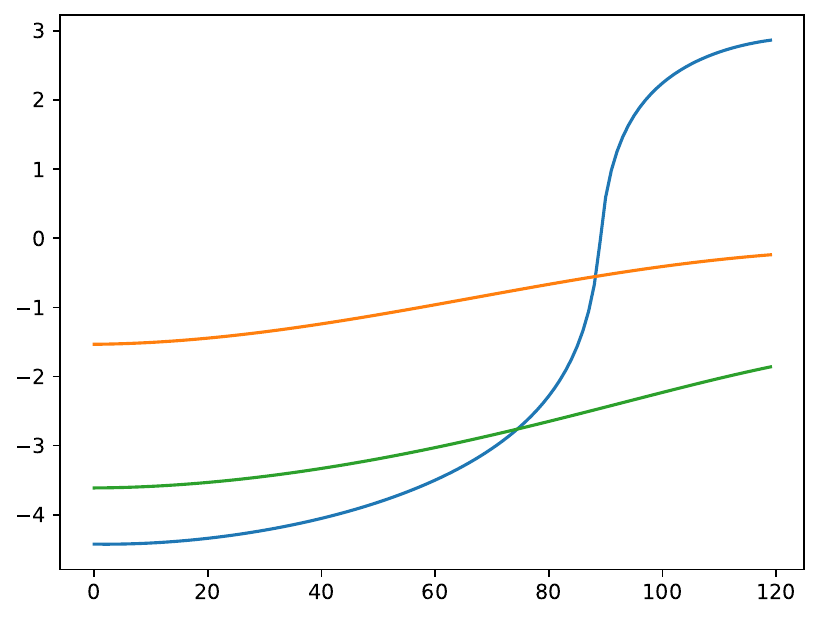}
  \caption{These images show the $g^m$ Gaussian curvatures at each point along the above Kendall geodesics in the space of triangles. The left panel corresponds to the top path and the right panel corresponds with the lower path. Note we omit the Kendall curvature as that would be a constant function $\equiv 1$ on either graph. The curvatures displayed with a symmetric log scale with green corresponding to $m=0$, orange corresponding to $m=1$, and blue corresponding to $m=2$. }
  \label{fig:geocurves1}
\end{figure}

\appendix
\section{Constant Coefficient Sobolev Metrics}\label{sec:appendix:constant}
Our use of the $\ell(c)$ terms in $\sgsi^m$ metrics allowed for scale invariance. These terms may be removed to arrive at the constant coefficient Sobolev metrics
\begin{equation*}\label{constcoefsmoothmetricdef}
  \sgcc_c^m(h,k):=\int_{S^1}\langle h,k\rangle+\langle D_s^m h,D_s^m k\rangle ds.
\end{equation*}
 These lack scale invariance, but the manifolds $(\Imm(S^1,\R^d),\sgcc^m)$ are otherwise similar to the scale invariant formulation with identical completeness properties \cite{bauer2015metrics, bruveris2015completeness, bruveris2014geodesic}. We can likewise drop the $\ell(c)$ terms from our discretized metrics to arrive at constant coefficient discrete metrics $\gcc^m$. These also converge to the corresponding smooth $\sgcc^m$ like the scale invariant formulation, using exactly the same proof as used for Proposition \ref{convergenceproposition}, but without the $\ell(c)$ terms in $F$. Similarly, the parallel of the completeness statement, Theorem \ref{mainbit} still holds, though one extra estimate given below is needed for the proof. The item most complicated by the removal of the length terms is the relation between $\dot{\gcc}^m_c(h,h)$ and $\dot{\gcc}^{m+1}_c(h,h)$. As $\dot{\gsi}^m_c=\dot{\gcc}^m_c\cdot\ell(c)^{-3+2m}$ the statement of Lemma \ref{metricequivalence} here becomes
$$\dot{\gcc}^m_c(h,h)\le \frac{\ell(c)^2}{4}\dot{\gcc}^{m+1}_c(h,h)$$
which is a parallel to the equivalent statement for the smooth constant coefficient case as seen in the proof of Lemma 2.13 of \cite{bruveris2014geodesic}. 
The estimate given in the proof of Lemma~\ref{BoundedOnMetricBalls} using Lemma \ref{lem:usefulreference} will still work if we may show the following for $m\ge 2$:
  \begin{claim}
  $$C=\sup_{c\in B_{\widetilde{g}^m}(\hat{c}_0,r)}\ell(c)^{\left(\frac{3}{2}-m\right)}<\infty.$$
  on each $\widetilde{g}^m$-metric ball $B_{\widetilde{g}^m}(\hat{c}_0,r)$.
\end{claim}
The rest of the proof survives with no other changes.
\begin{proof}
  We can show $C<\infty$ by application of Lemma \ref{lem:usefulreference} and showing $\ell(c)^{\left(\frac{3}{2}-m\right)}$ is Lipschitz on $\gcc^m$ metric ball. We estimate to apply Lemma~\ref{lem:usefulreference} to $\ell(c)^{\left(\frac{3}{2}-m\right)}$: $$\left|\frac{d}{dh}\ell(c)^{\left(\frac{3}{2}-m\right)}\right|\le\frac{1}{2}\ell(c)^{\left(\frac{1}{2}-m\right)}\left(\sum_{i=1}^n\frac{\|h_{i+1}-h_i\|^2}{|e_i|}\right)^{\frac{1}{2}}\left(\sum_{i=1}^n|e_i|\right)^{\frac{1}{2}}$$ $$=\frac{\ell(c)^{\left(1-m\right)}}{2}\left(\sum_{i=1}^n\frac{\|h_{i+1}-h_i\|^2}{|e_i|}\right)^{\frac{1}{2}}\le\frac{\ell(c)^{\left(1-m\right)}}{2}\frac{\ell(c)^{\left(m-1\right)}}{2^{m-1}}\sqrt{\dot{\gcc}^m(h,h)}\le \frac{1}{2^{m}}\|h\|_{\gcc^{m}}.$$
\end{proof}

\section{Approximating derivatives in $L^{\infty}$}
\begin{lemma}\label{derivativeapproximators}
  Let $f$ be a function in $C^\infty(S^1)$. Let $\theta_i=\frac{i}{n}$ and define 
  $$I_n^0:C^\infty(S^1)\to L^\infty(S^1)$$
  so that $I_n^0(f)=g_0$ where $g_0$ is the piecewise constant function where $g_0(\theta)=f(\theta_i)$ for $\theta\in[\theta_i,\theta_{i+1})$. Further define $I_n^1(f)=g_1$ where $g_1(\theta)=n(f(\theta_{i+1})-f(\theta_i))$ for $\theta\in[\theta_{i},\theta_{i+1})$. Then
  $$\lim_{n\to\infty}\|I_n^0(f)-f\|_{L^\infty}=0\qquad\text{and}\qquad\lim_{n\to\infty}\|I_n^1(f)-f'\|_{L^\infty}=0.$$
\end{lemma}
\begin{proof}
  Note that $S^1$ is compact so that our smooth, necessarily bounded function $f$ is Lipschitz continuous. Now if $K$ is our Lipschitz constant then $0\le \sup_{\theta\in[\theta_i,\theta_{i+1})}|g_0(\theta)-f(\theta)|\le \frac{K}{n}$ but this does not depend on $i$. 
  Thus
  $$\lim_{n\to\infty}\|I_n^0(f)-f\|\le \lim_{n\to\infty}\frac{K}{n}\to 0.$$
  For the second operator we make use of the mean value theorem
  \begin{align*}
      \lim_{n\to\infty}\|I_n^1(f)-f'\|_{L^\infty}&=\lim_{n\to\infty}\max_{i=1,\ldots,n}\|n(f(\theta_{i+1})-f(\theta_i))-f'\|_{L^\infty}\\
      &=\lim_{n\to\infty}\max_{i=1,\ldots,n}\left\|n\int_{\theta_i}^{\theta_{i+1}}f'(\theta)d\theta-f'\right\|_{L^\infty}\\
      &=\lim_{n\to\infty}\max_{i=1,\ldots,n}\left\|n\frac{f'(c_i)}{n}-f'\right\|_{L^\infty}\le\lim_{n\to\infty}\frac{K}{n}=0
  \end{align*}
\end{proof}
Define $\widetilde{I}_n^1(f)$ analogously to $I_n^1(f)$, but with $g_1(\theta)=n\cdot (f(\theta_{i})-f(\theta_{i-1}))$ for $\theta\in[\theta_{i},\theta_{i+1})$ and note that this is also such that $\lim_{n\to\infty}\|\widetilde{I}_n^1(f)-f'\|_{L^\infty}=0$.\\

\section{Proof of Lemma~\ref{metricequivalence}}
Before we are able to present the proof of Lemma~\ref{metricequivalence} we will need an additional technical Lemma which is a  discrete Poincar\'e type inequality, but adapted to the discrete arclength derivatives used in this work:
\begin{lemma}\label{clm:intermediatetwo}
  Let $m\ge 1$, let $c\in \Rdnst$ and $h\in T_c\Rdnst\cong \Rdn$, we then have
  \begin{equation}
    \max_{k\in\mathbb{Z}/n\mathbb{Z}}\left|\frac{D_s^{m-1}(h_{k+1})-D_s^{m-1}(h_k)}{|e_k|}\right|\leq \frac{1}{2}\sum_{i=1}^{n}\left|\frac{D_s^{m-1}(h_{i+1})-D_s^{m-1}(h_i)}{|e_i|}-\frac{D_s^{m-1}(h_i)-D_s^{m-1}(h_{i-1})}{|e_{i-1}|}\right|
  \end{equation}
 for odd $m$ and
  \begin{equation}
    \max_{k\in\mathbb{Z}/n\mathbb{Z}}\left|\frac{D_s^{m-1}(h_k)-D_s^{m-1}(h_{k-1})}{\mu_k}\right|\leq \frac{1}{2}\sum_{i=1}^{n}\left|\frac{D_s^{m-1}(h_{i+1})-D_s^{m-1}(h_i)}{\mu_{i+1}}-\frac{D_s^{m-1}(h_i)-D_s^{m-1}(h_{i-1})}{\mu_{i}}\right|
  \end{equation}
  for even $m$.
\end{lemma}
\begin{proof}
  This Lemma follows by direct estimation. Therefore, let $k\in\mathbb{Z}/n\mathbb{Z}$ and note that $\sum_{j=1}^{n}(D_s(h_{j+1})-D_s(h_{j}))=0$ as each term appears in the sum both in positive form and negative. Then, for odd $m$,
  {
    \allowdisplaybreaks
    \begin{align*}
      &\left|\frac{D_s^{m-1}(h_{k+1})-D_s^{m-1}(h_k)}{|e_k|}\right|\\
      &= \left|\frac{1}{\ell(c)}\sum_{j=1}^{n}(D_s^{m-1}(h_{j+1})-D_s^{m-1}(h_{j}))-\frac{D_s^{m-1}(h_{k+1})-D_s^{m-1}(h_k)}{|e_k|}\right|\\
      &=\left|\frac{1}{\ell(c)}\sum_{j=1}^{n}(D_s^{m-1}(h_{j+1})-D_s^{m-1}(h_{j}))-\frac{1}{\ell(c)}\sum_{j=1}^{n}\frac{D_s^{m-1}(h_{k+1})-D_s^{m-1}(h_k)}{|e_k|}|e_j|\right|\\
      &\leq\left|\frac{1}{\ell(c)}\sum_{j=1}^{n}\left(\frac{D_s^{m-1}(h_{j+1})-D_s^{m-1}(h_{j})}{|e_j|}-\frac{D_s^{m-1}(h_{k+1})-D_s^{m-1}(h_k)}{|e_k|}\right)|e_j|\right|\\
      &\leq\left|\frac{1}{\ell(c)}\sum_{j=1}^{n}\frac{1}{2}\left(\sum_{i=k}^{j-1} \left(\frac{D_s^{m-1}(h_{i+i})-D_s^{m-1}(h_i)}{|e_i|}-\frac{D_s^{m-1}(h_i)-D_s^{m-1}(h_{i-1})}{|e_{i-1}|}\right)\right.\right.\\
      &\left.\left.\qquad\qquad\qquad-\sum_{i=j}^{k-1} \left(\frac{D_s^{m-1}(h_{i+1})-D_s^{m-1}(h_i)}{|e_i|}-\frac{D_s^{m-1}(h_i)-D_s^{m-1}(h_{i-1})}{|e_{i-1}|}\right)\right)|e_j|\right|\\
      &\leq \frac{1}{2\ell(c)}\sum_{j=1}^n\sum_{i=1}^n\left|\frac{D_s^{m-1}(h_{i+1})-D_s^{m-1}(h_i)}{|e_i|}-\frac{D_s^{m-1}(h_i)-D_s^{m-1}(h_{i-1})}{|e_{i-1}|}\right||e_j|\\
      &\leq\frac{1}{2}\sum_{i=1}^n\left|\frac{D_s^{m-1}(h_{i+1})-D_s^{m-1}(h_i)}{|e_i|}-\frac{D_s^{m-1}(h_{i})-D_s^{m-1}(h_{i-1})}{|e_{i-1}|}\right|
    \end{align*}
  }
  The case for even $m$ is similar, swapping $|e_i|$ for $\mu_i$ and rotating some indices.
  {
    \allowdisplaybreaks
    \begin{align*}
      &\left|\frac{D_s^{m-1}(h_{k})-D_s^{m-1}(h_{k-1})}{\mu_k}\right|\\
      &= \left|\frac{1}{\ell(c)}\sum_{j=1}^{n}(D_s^{m-1}(h_{j})-D_s^{m-1}(h_{j-1}))-\frac{D_s^{m-1}(h_{k})-D_s^{m-1}(h_{k-1})}{\mu_k}\right|\\
      &=\left|\frac{1}{\ell(c)}\sum_{j=1}^{n}(D_s^{m-1}(h_{j})-D_s^{m-1}(h_{j-1}))-\frac{1}{\ell(c)}\sum_{j=1}^{n}\frac{D_s^{m-1}(h_{k})-D_s^{m-1}(h_{k-1})}{\mu_k}\mu_j\right|\\
      &\leq\left|\frac{1}{\ell(c)}\sum_{j=1}^{n}\left(\frac{D_s^{m-1}(h_{j})-D_s^{m-1}(h_{j-1})}{\mu_j}-\frac{D_s^{m-1}(h_{k})-D_s^{m-1}(h_{k-1})}{\mu_k}\right)\mu_j\right|\\
      &\leq\left|\frac{1}{\ell(c)}\sum_{j=1}^{n}\frac{1}{2}\left(\sum_{i=k}^{j-1} \left(\frac{D_s^{m-1}(h_{i+i})-D_s^{m-1}(h_i)}{\mu_{i+1}}-\frac{D_s^{m-1}(h_i)-D_s^{m-1}(h_{i-1})}{\mu_{i}}\right)\right.\right.\\
      &\left.\left.\qquad\qquad\qquad-\sum_{i=j}^{k-1} \left(\frac{D_s^{m-1}(h_{i+1})-D_s^{m-1}(h_i)}{\mu_{i+1}}-\frac{D_s^{m-1}(h_i)-D_s^{m-1}(h_{i-1})}{\mu_{i}}\right)\right)\mu_j\right|\\
      &\leq \frac{1}{2\ell(c)}\sum_{j=1}^n\sum_{i=1}^n\left|\frac{D_s^{m-1}(h_{i+1})-D_s^{m-1}(h_i)}{\mu_i}-\frac{D_s^{m-1}(h_i)-D_s^{m-1}(h_{i-1})}{\mu_{i-1}}\right|\mu_j\\
      &\leq\frac{1}{2}\sum_{i=1}^n\left|\frac{D_s^{m-1}(h_{i+1})-D_s^{m-1}(h_i)}{\mu_{i+1}}-\frac{D_s^{m-1}(h_{i})-D_s^{m-1}(h_{i-1})}{\mu_{i}}\right|
    \end{align*}
  }
  This proves the lemma.
\end{proof}

\begin{proof}[Proof of Lemma~\ref{metricequivalence}]
  We consider the odd case which is not meaningfully different from the even one. Following from Lemma \ref{clm:intermediatetwo}, we have
  {\allowdisplaybreaks
    \begin{align*}
      &\max_{i\in \mathbb{Z}/n\mathbb{Z}}\left|\frac{D^{m-1}(h_{k+1})-D^{m-1}(h_k)}{|e_k|}\right|^2\\&\leq \frac{1}{4}\left(\sum_{i=1}^{n}\left|\frac{D_s^{m-1}h_{i+1}-D_s^{m-1}(h_i)}{|e_{i}|}-\frac{D_s^{m-1}(h_i)-D_s^{m-1}(h_{i-1})}{|e_{i-1}|}\right|\right)^2\\
      &\leq \frac{1}{4}\left(\sum_{i=1}^{n}\left|\frac{D_s^{m-1}(h_{i+1})-D_s^{m-1}(h_i)}{|e_{i}|}-\frac{D_s^{m-1}(h_i)-D_s^{m-1}(h_{i-1})}{|e_{i-1}|}\right|^2 \frac{1}{\mu_i}\right)\left(\sum_{i=1}^{n} \mu_i \right)\\
      &=\frac{\ell(c)}{4}\left(\sum_{i=1}^{n}\left|\frac{D_s^{m-1}(h_{i+1})-D_s^{m-1}(h_i)}{|e_{i}|}-\frac{D_s^{m-1}(h_i)-D_s^{m-1}(h_{i-1})}{|e_{i-1}|}\right|^2 \frac{1}{\mu_i}\right).
    \end{align*}
  }
  Then, due to H\"older's inequality and the above, we conclude
  {
    \allowdisplaybreaks
    \begin{align*} \dot{\gsi}_c^m(h,h)&=\sum_{i=1}^n\ell(c)^{-3+2m}|D_s^m(h_i)|^2|e_i|\\
                                      &=\sum_{i=1}^n\ell^{-3+2m}\frac{|D_s^{m-1}(h_{i+1})-D_s^{m-1}(h_{i})|^2}{|e_i|}\\
                                      &\leq\ell(c)^{-3+2m}\left(\max_{k\in \mathcal{C}_n}\left|\frac{D_s^{m-1}(h_{k+1})-D_s^{m-1}(h_k)}{|e_k|}\right|^2\right)\left(\sum_{k=1}^{n}|e_k|\right)\\
                                      &\leq\ell(c)^{-3+2m}\frac{\ell(c)^2}{4}\sum_{i=1}^{n}\left|\frac{D_s^{m-1}(h_{i+1})-D_s^{m-1}(h_i)}{|e_{i}|}-\frac{D_s^{m-1}(h_i)-D_s^{m-1}(h_{i-1})}{|e_{i-1}|}\right|^2 \frac{1}{\mu_i}\\
                                      &=\frac{\ell(c)^{-3+2(m+1)}}{4}\sum_{i=1}^n|D_s^{m+1}(h_i)|^2\mu_i=\frac{1}{4}\dot{\gsi}_c^{m+1}(h,h).
    \end{align*}
  }
  This is parallel to the argument for the odd case, which we omit.
\end{proof}

\bibliographystyle{abbrv}
\bibliography{references.bib}
\end{document}